\documentclass[11pt]{article}

\usepackage[T1]{fontenc}
\usepackage[utf8]{inputenc}
\usepackage{amsmath,amssymb,amsthm,mathtools}
\usepackage{microtype}
\usepackage{xcolor}
\usepackage{hyperref}
\usepackage{fullpage}

\newtheorem{theorem}{Theorem}
\newtheorem{lemma}[theorem]{Lemma}

\newcommand{\E}{\mathbb E}
\newcommand{\R}{\mathbb R}
\renewcommand{\epsilon}{\varepsilon}

\DeclarePairedDelimiter{\abs}{\lvert}{\rvert}

\title{The Last Iterate of the Projected Stochastic Subgradient Method\\Pays a Logarithm, Even in Dimension One}

\author{Guglielmo Beretta\thanks{Department of Computer Science, Aarhus University, Aarhus, Denmark.}
\and Tommaso Cesari\thanks{School of Electrical Engineering and Computer Science, University of Ottawa, Ottawa, Canada.}
\and Roberto Colomboni\thanks{School of Mathematics, University of Bristol, Bristol, United Kingdom.}}

\begin{document}
\maketitle

\begin{abstract}
We prove that the projected stochastic subgradient method with the standard anytime stepsizes $\eta_t = \gamma/\sqrt t$ can suffer expected last-iterate optimization error of order $(\log n)/\sqrt n$ already on the real line.
Combined with the classical in-expectation upper bound, this establishes the sharp worst-case order of the projected stochastic subgradient method.
\end{abstract}

\section{Introduction}

The projected stochastic subgradient method is a classical algorithm for minimizing a convex objective function. Let $\mathcal X \subset \mathbb{R}^d$ be a nonempty closed convex subset, and let $f\colon\mathcal X\to\mathbb R$ be a convex function with a nonempty set of minimizers.
If the stochastic first-order oracle is conditionally unbiased and has uniformly bounded conditional second moments, with conditioning taken with respect to the history prior to each oracle call, then, when the anytime stepsizes $\eta_t = \gamma/\sqrt{t}$ are used, the expected optimization error of the last iterate is known to be $O\bigl(\log n)/\sqrt{n}\bigr)$ (see, for instance, Orabona~\cite[Theorem~2 and Corollary~3]{Orabona2020}).

In this note, we show that this rate is worst-case optimal already on the real line.
More precisely, for every $b$, $\gamma>0$ and every $n \ge 2$, there exist a convex function $f\colon\R\to\R$, minimized at $0$ and with Lipschitz constant at most $b/2$, and an additive noise process consisting of i.i.d.\ scaled Rademacher random variables with magnitude at most $b$, such that, when the stochastic subgradient method is initialized at $X_1 \coloneq 0$ and run with stepsizes $\eta_t=\gamma/\sqrt{t}$, the resulting stochastic subgradients are bounded by $b$ and its $n$-th iterate satisfies
\[
    \E\left[f(X_n)-\min_{x\in\R}f(x)\right]
\ge
    \frac{\gamma b^2}{24} \cdot \frac{\log n}{\sqrt{n}}.
\]
Thus, the classical $O\bigl((\log n)/\sqrt{n}\bigr)$ upper bound is sharp even in dimension one under additive, state-independent, uniformly bounded i.i.d.\ noise.

We remark that this behavior is in sharp contrast with the fixed-horizon setting.
Indeed, for constant stepsizes
\[
    \eta_1
=
    \dots
=
    \eta_n
=
    \gamma/\sqrt{n},
\]
Beretta et al.~\cite{BerettaEtAl2026} recently proved that, in dimension one and under additive i.i.d.\ subgradient noise with uniformly bounded variance, the last iterate achieves the optimal $O(n^{-1/2})$ rate, without the logarithmic factor.
Thus, even under particularly benign stochastic noise, the logarithmic loss may arise from the standard anytime schedule itself rather than from the noise assumptions alone.

\section{The scalar bounded-i.i.d.\ lower bound}\label{sec:lower}

Fix $\gamma>0$ and, for all $t \in \mathbb{N}$, define the stepsizes
\[
    \eta_t
\coloneq
    \frac{\gamma}{\sqrt t}.
\]
Given a convex function $f\colon\R\to\R$, a measurable selection $g(x)\in\partial f(x)$,
and a sequence of random variables $(W_t)_{t \ge 1}$, consider the recurrence
\begin{equation}\label{eq:scalar-update}
    X_{t+1}
=
    X_t-\eta_t\bigl(g(X_t)+W_t\bigr),
\qquad
    X_1
=
    0.
\end{equation}
Let
\[
    \mathcal F_t
\coloneq
    \sigma(X_1,W_1,\ldots,W_{t-1}).
\]
If $(W_t)_{t \ge 1}$ is a centered i.i.d.\ sequence, then $X_t$ is
$\mathcal F_t$-measurable and
\[
    \E\bigl[g(X_t)+W_t\mid\mathcal F_t\bigr]
=
    g(X_t)\in\partial f(X_t),
\]
so that \eqref{eq:scalar-update} is the recurrence of the stochastic subgradient method.

We recall that a Rademacher random variable is a random variable $\xi$ whose range is $\{-1,1\}$ for which $\mathbb{P} (\xi = -1) = 1/2 = \mathbb{P} (\xi = 1)$.
In the proof of our result, we use the following form of the Khintchine inequality \cite{Szarek1976}.

\begin{lemma}[Khintchine-Szarek]\label{lem:khintchine}
If $\xi_1$, $\dots$, $\xi_m$ are independent Rademacher random variables and
$a_1$, $\dots$, $a_m \in \R$, then
\[
    \frac{1}{\sqrt{2}} \left(\sum_{t=1}^m a_t^2\right)^{1/2}
\le
    \E\left[\abs*{\sum_{t=1}^m a_t \xi_t }\right]
\le
    \left(\sum_{t=1}^m a_t^2\right)^{1/2}.
\]
\end{lemma}

We are now ready to formally state and prove our lower bound.

\begin{theorem}\label{thm:lower}
Let $b>0$. For every integer $n \ge 2$, there exist a convex function
$f\colon\R\to\R$, a Borel selection $g(x)\in\partial f(x)$, and a sequence of
random variables $(W_t)_{t \ge 1}$ such that:
\begin{enumerate}
    \item $f$ is uniquely minimized at $0$;
    \item $f$ is Lipschitz continuous with Lipschitz constant at most $b/2$;
    \item $(W_t)_{t \ge 1}$ is an i.i.d.\ sequence of scaled Rademacher random
    variables with magnitude at most $b$;
    \item the iterates generated by \eqref{eq:scalar-update} satisfy
    \[
        \bigl\lvert g(X_t)+W_t \bigr\rvert
    \le
        b
    \qquad
        \text{for every $t$},
    \]
    and
    \[
        \E\left[f(X_n)-\min_{x\in\R}f(x)\right]
    \ge
        \frac{\gamma b^2}{24} \cdot \frac{\log n}{\sqrt n}.
    \]
\end{enumerate}
\end{theorem}

\begin{proof}
Let $(\xi_t)_{t \ge 1}$ be an i.i.d.\ Rademacher sequence.
Let $s\coloneq\sum_{t=1}^{n-1}\frac{1}{\sqrt t}$.
Let $\mu\coloneq\E \left[\left|\sum_{t=1}^{n-1}\frac{\xi_t}{\sqrt t}\right|\right]$.
Define $\epsilon \coloneq \frac{b}{2} \cdot \frac{\mu}{\mu+s}$ and notice that $0 < \epsilon < b/2$.
Define
\[
    f(x)
\coloneq
    \epsilon |x|
\]
and the measurable subgradient selection
\[
    g(x)
\coloneq
    \begin{cases}
        \epsilon\operatorname{sign}(x), & x\ne0,\\
        0, & x=0.
    \end{cases}
\]
Then $f$ is convex, uniquely minimized at $0$, and $\epsilon$-Lipschitz.
Set
\[
    W_t
\coloneq
    (b-\epsilon)\xi_t.
\]
The variables $(W_t)_{t \ge 1}$ are i.i.d.\ scaled Rademacher random variables, and
\[
    \bigl\lvert g(X_t)+W_t \bigr\rvert
\le
    \bigl\lvert g(X_t) \bigr\rvert +|W_t|
\le
    \epsilon+(b-\epsilon)
=
    b.
\]
Notice that, unrolling \eqref{eq:scalar-update}, we have
\[
    X_n
=
    -\gamma(b-\epsilon) \sum_{t=1}^{n-1}\frac{\xi_t}{\sqrt t} -\gamma \sum_{t=1}^{n-1}\frac{g(X_t)}{\sqrt t},
\]
which in turn gives
\begin{align*}    
    \bigl\lvert f(X_{n}) \bigr\rvert
&=
    \epsilon\abs{X_n}
=
    \gamma \epsilon \cdot \abs*{-(b-\epsilon)\sum_{t=1}^{n - 1} \frac{\xi_t}{\sqrt{t}} - \sum_{t=1}^{n - 1} \frac{g(X_t)}{\sqrt{t}}}
\ge 
    \gamma \epsilon \cdot \left(\abs*{-(b-\epsilon)\sum_{t=1}^{n - 1} \frac{\xi_t}{\sqrt{t}}} - \abs*{\sum_{t=1}^{n - 1} \frac{g(X_t)}{\sqrt{t}}}\right)
\\
&=
    \gamma \epsilon \cdot \left( (b - \epsilon) \cdot \abs*{\sum_{t=1}^{n - 1} \frac{\xi_t}{\sqrt{t}}} - \abs*{\sum_{t=1}^{n - 1} \frac{g(X_t)}{\sqrt{t}}}\right)
\ge
    \gamma \epsilon \cdot \left( (b - \epsilon) \cdot \abs*{\sum_{t=1}^{n - 1} \frac{\xi_t}{\sqrt{t}}} - \epsilon s \right),
\end{align*}
where we used the triangle inequality, $0 < \epsilon < b/2$, and that $\abs{g(x)} \le \epsilon$ for every $x \in \R$.
Taking expectations yields
\[
    \E\bigl[f(X_n)\bigr]
\ge
    \gamma\epsilon \cdot \bigl((b - \epsilon) \mu-\epsilon s \bigr)
=
    \gamma\epsilon \cdot \bigl(b\mu-\epsilon(\mu+s)\bigr).
\]
and substituting the value of $\epsilon$ gives
\[
    \E\bigl[f(X_n)\bigr]
\ge
    \frac{\gamma b^2}{4} \cdot \frac{\mu^2}{\mu+s}.
\]
Set
$
    h
\coloneq
    \sum_{t=1}^{n-1}\frac{1}{t}
$.
By Lemma~\ref{lem:khintchine},
\[
    \mu^2
\ge
    \frac{h}{2}
\qquad
    \text{and}
\qquad
    \mu
\le
    \sqrt h,
\]
and therefore
\begin{equation}\label{eq:technical-lower}
    \E\bigl[f(X_n)\bigr]
\ge
    \frac{\gamma b^2}{8} \cdot \frac{h}{s+\sqrt h}.
\end{equation}
It remains to convert \eqref{eq:technical-lower} into the claimed logarithmic bound.
For every $n \ge 2$, the elementary estimates
\[
    h
\ge
    \int_1^n \frac{\mathrm{d}x}{x}
=
    \log n,
\]
\[
    s
\le
    1+\int_1^{n-1}\frac{\mathrm{d}x}{\sqrt{x}}
=
    2\sqrt{n-1}-1
\le
    2\sqrt n,
\]
and $h \le n-1\le n$ directly implies
\[
    \frac{h}{s+\sqrt h}
\ge
    \frac{\log n}{3\sqrt n}.
\]
Combining this with \eqref{eq:technical-lower} gives, for every $n \ge 2$,
\[
    \E\bigl[f(X_n)\bigr]
\ge
    \frac{\gamma b^2}{24} \cdot \frac{\log n}{\sqrt n},
\]
as claimed.
\end{proof}

\section*{Acknowledgments}
TC gratefully acknowledges the support of the Natural Sciences and Engineering Research Council of Canada (NSERC) through grant RGPIN-2023-03688 (Discovery Grants Program).

\bibliographystyle{plain}
\bibliography{biblio}

@misc{Orabona2020,
  author       = {Francesco Orabona},
  title        = {Last Iterate of {SGD} Converges (Even in Unbounded Domains)},
  year         = {2020},
  month        = aug,
  howpublished = {\url{https://parameterfree.com/2020/08/07/last-iterate-of-sgd-converges-even-in-unbounded-domains/}},
  url          = {https://parameterfree.com/2020/08/07/last-iterate-of-sgd-converges-even-in-unbounded-domains/},
  note         = {Parameter-free Learning and Optimization Algorithms, blog post, 7 August 2020}
}

@article{Szarek1976,
  author  = {Stanis{\l}aw J. Szarek},
  title   = {On the Best Constants in the {Khinchin} Inequality},
  journal = {Studia Mathematica},
  volume  = {58},
  number  = {2},
  pages   = {197--208},
  year    = {1976},
  doi     = {10.4064/sm-58-2-197-208},
  url     = {https://doi.org/10.4064/sm-58-2-197-208}
}

@misc{BerettaEtAl2026,
  author        = {Guglielmo Beretta and Tommaso Cesari and Roberto Colomboni and Andrea Paudice},
  title         = {New Bounds for the Last Iterate of the {Stochastic subGradient Method}},
  note       = {arXiv preprint {\tt arXiv:2606.24879}},
  year          = {2026},
  eprint        = {2606.24879},
  archivePrefix = {arXiv},
  primaryClass  = {math.OC}
}

\end{document}